\documentclass[final,leqno,onefignum,onetabnum,dvipsnames]{siamltex1213}
\usepackage[normalem]{ulem}
\usepackage{bbm}
\usepackage{framed}
\usepackage{amssymb}
\usepackage{graphicx}
\usepackage[outdir=./]{epstopdf}
\usepackage{float}
\usepackage{caption}
\usepackage{sidecap}
\usepackage{wrapfig}
\usepackage{mathptmx}
\RequirePackage{ifthen}
\usepackage{times}
\usepackage{amsmath}
\usepackage{algpseudocode}
\usepackage{enumerate}
\usepackage{color}
\usepackage{xcolor}
\usepackage{mathtools}
\DeclarePairedDelimiter\ceil{\lceil}{\rceil}
\DeclarePairedDelimiter\abs{\lvert}{\rvert}
\DeclarePairedDelimiter\norm{\lVert}{\rVert}

\title{Multilevel Ensemble Transform Particle Filtering } 

\author{A. Gregory \footnotemark[2], C. J. Cotter \footnotemark[2] and S. Reich \footnotemark[3]} 

\begin{document}
\maketitle

\renewcommand{\thefootnote}{\fnsymbol{footnote}}

\footnotetext[2]{Department of Mathematics, Imperial College London, Exhibition Road, London SW7 2AZ, UK (a.gregory14@imperial.ac.uk). Alastair Gregory was supported by the Science and Solutions to a Changing Planet DTP and the Natural Environmental Research Council.}
\footnotetext[3]{Institut f{\"u}r Mathematik, Universit{\"a}t Potsdam, Am Neuen Palais 10, D-14469 Potsdam, Germany and Department of Mathematics and Statistics, University of Reading, Whiteknights, Reading RG6 6AX, UK.}

\renewcommand{\thefootnote}{\arabic{footnote}}

\slugger{mms}{xxxx}{xx}{x}{x--x}

\begin{abstract}
This paper extends the Multilevel Monte Carlo variance reduction technique to nonlinear filtering. In particular, Multilevel Monte Carlo is applied to a certain variant
of the particle filter, the Ensemble Transform Particle Filter. A key
aspect is the use of optimal transport methods to re-establish
correlation between coarse and fine ensembles after resampling; this controls
the variance of the estimator.
Numerical examples present a proof of concept of the effectiveness of
the proposed method, demonstrating significant computational cost
reductions (relative to the single-level ETPF counterpart) in the
propagation of ensembles.
\end{abstract}



\pagestyle{myheadings}
\thispagestyle{plain}
\markboth{A. GREGORY, C. J. COTTER AND S. REICH}{MULTILEVEL ENSEMBLE PARTICLE TRANSFORM FILTERING}

\section{Introduction}

Data assimilation is the process of incorporating observed data into
model forecasts. In data assimilation, one is interested in computing
statistics $\mathbb{E}_{\eta}[X]$ of solutions $X$ to random dynamical
systems with respect to a posterior measure ($\eta$) given partial
observations of the system. In particle filtering \cite{Doucet,Cappe}, this
is done by using an empirical ensemble representing the posterior
distribution $\eta$ at any one time. The propagation in time of the
members (particles) of this ensemble can be computationally expensive,
especially in high dimensional systems.

Recently, the Multilevel Monte Carlo (MLMC) method has been developed
for achieving significant cost reductions in Monte Carlo simulations
\cite{Giles}. It has been applied to areas such as Markov Chain Monte Carlo \cite{Teckentrup} and quasi-Monte Carlo \cite{GilesQuasi} to return computational cost reductions from exisiting techniques. It has also been applied to uncertainty quantification within PDEs \cite{Cliffe}.
The idea is to consider a hierarchy of discretized
models, balancing numerical error in cheap/coarse models against Monte
Carlo variance in expensive/fine models. It is desirable to adapt
MLMC to sequential Monte Carlo methods such as particle filters, and
some first steps have been taken in this direction. Firstly,
\cite{Hoel} have developed a multilevel Ensemble Kalman Filter
(EnKF), using MLMC estimators to calculate the mean and covariance
of the posterior, in the case where the underlying distributions are
Gaussian and the model is linear. However for non-Gaussian
distributions and nonlinear models, the EnKF is biased. The method does however converge to a ``mean-field limit'' \cite{Legland}. Secondly,
\cite{Beskos} proposed a multilevel sequential Monte Carlo method for
Bayesian inference problems to give significant computational cost
reductions from standard techniques. Our goal in this paper is to take
a step further by applying MLMC to nonlinear filtering problems.

In general, MLMC works by computing statistics from pairs of coarser
and finer correlated ensembles. For Monte Carlo simulation of SDEs,
this correlation is achieved by using the same initial conditions and
Brownian paths for each coarse/fine pair of ensemble members. The key
challenge in applying MLMC to particle filtering is in maintaining
this correlation after resampling. \cite{GilesBook} suggested that
correlating coarse and fine ensembles could be achieved by minimising
the Wasserstein distance between the two ensembles. This can formulated
as a optimal transportation problem \cite{EnsembleForecasting}.

In this paper, we adapt the MLMC framework to the Ensemble Transform
Particle Filter (ETPF) \cite{Reich}. ETPF is an efficient and
effective nonlinear filter that uses optimal transportation
transformations \cite{Villani} instead of random resampling.  In our
Multilevel ETPF (MLETPF) the coupling between coarse and fine
ensembles is also maintained using optimal transportation.  The sole
aim of introducing MLMC to the ETPF is to reduce the computational
cost of the propagation of particles.  This is only a benefit if the
computational cost dominates the optimal transportation transformation
cost; whilst direct solvers for optimal transportation problems with
one-dimensional state space scale as $O\big(Nlog(N)\big)$, solvers for problems with
more than one dimension scale as $O\big(N^3log(N)\big)$ with the ensemble size. To
address this, a technique commonly used in the aforementioned EnKF
known as localisation can be used to reduce this optimal
transportation cost significantly \cite{Cheng}. Our proposed MLETPF
can return significant reductions in the overall computational cost of
ETPF where the particle propagation cost dominates. It will also
return significant reductions in cases where optimal transportation
computational cost dominates, if the localised ETPF is used.

This paper proceeds as follows. Section \ref{sec:MLMC} provides a
background of the MLMC method, \S \ref{sec:PF} describes the
basic particle filtering framework together with the ETPF
scheme. Then, the proposed Multilevel ETPF (MLETPF) method is
presented in \S \ref{sec:MLETPF}, along with numerical examples to demonstrate the
effectiveness of the method. Finally, \S \ref{sec:summary}
provides a summary and outlook.

\section{{The Multilevel Monte Carlo Method}}
\label{sec:MLMC}

The Multilevel Monte Carlo estimator can be viewed as a variance
reduction technique for a standard Monte Carlo estimator.  Suppose one
wishes to compute an approximation of $\mathbb{E}[X_{L}]$,
where $X_{L}$ is a numerical approximation of a random variable $X$
(with discretization accuracy parameter\footnote{Such as the time
  stepsize.} $h_{L} \propto M^{-L}$). The Multilevel Monte Carlo (MLMC) method
introduced in \cite{Giles} considered the case where $X$ is the
solution to a stochastic differential equation (SDE) at time $T>0$;
the discretized solutions $X_{L}$ are obtained from a given numerical
method with stepsize $h_L$. This paper will instead consider $X$ to be
a solution, at time $T>0$, to a general random dynamical system, with
stochastic forcing and/or random initial conditions (drawn from a
distribution $\pi^{0}$). In the simplest case let $X^{i}_{L}$,
$i=1,\ldots,N$, be $N \geq 1$ i.i.d. samples of $X_{L}$. 
The standard, unbiased, Monte Carlo estimator for $\mathbb{E}[X_{L}]$ is then
\begin{equation}
\bar{X}_{L}^{MC}=\frac{1}{N}\sum^{N}_{i=1}X_{L}^{i}.
\label{equation:MonteCarlo}
\end{equation}
Using a telescoping sum of expectations,
\begin{equation}
\mathbb{E}[X_{L}]=\mathbb{E}[X_{0}]+\sum^{L}_{l=1}\mathbb{E}[X_{l}]-\mathbb{E}[X_{l-1}],
\label{equation:linearityexpectation}
\end{equation}
one can define the MLMC approximation to $\mathbb{E}[X_{L}]$ as a sum of independent Monte Carlo estimators, $\bar{X}_{L}=\sum^{L}_{l=0}\hat{X}_{l}$, where
\begin{equation}
\hat{X}_{l}=
\begin{cases}
\sum^{N_{0}}_{i=1}\frac{1}{N_{0}}X_{0}^{i},& l=0, \\
\sum^{N_{l}}_{i=1}\frac{1}{N_{l}}\big(X_{l}^{i}-X_{l-1}^{i}\big),& l>0,
\end{cases}
\label{equation:MLMC}
\end{equation}
leading to
\begin{equation}
\bar{X}_{L}=\frac{1}{N_{0}}\sum^{N_{0}}_{i=1}X_{0}^{i}+\sum^{L}_{l=1}\Big(\frac{1}{N_{l}}\sum^{N_{l}}_{i=1}(X_{l}^{i}-X_{l-1}^{i})\Big).
\label{equation:MLestimator}
\end{equation}
Here, $N_{l}$, $l=0,...,L$, are Monte Carlo sample sizes, of i.i.d. draws from $X_{l},X_{l-1}$ respectively, for each of the $L+1$ estimators. We have
\begin{equation}
\mathbb{E}[\hat{X}_{l}]=
\begin{cases}
\mathbb{E}[X_{0}],& l=0, \\
\mathbb{E}[X_{l}]-\mathbb{E}[X_{l-1}],& l>0 ,
\end{cases}
\label{equation:unbiased}
\end{equation}
and hence the MLMC estimator is an unbiased approximation of
$\mathbb{E}[X_{L}]$. We will call the estimators,
$\hat{X}_{l}$, $l>0$, `multilevel difference estimators'.  The
important thing to note here is that the fine (level $l$) and coarse
(level $l-1$) samples in each difference estimator must be positively
correlated for each $i$ in each of the $L+1$ multilevel difference
estimators. This can be achieved by using the same random system input
(e.g. initial conditions/stochastic forcing) for each $i$ on both
levels. On the other hand, the samples in different difference
estimators must be uncorrelated.

For fixed $T$, the discretization bias (away from $\mathbb{E}[X]$) of the overall estimator is $O(h_{L}^{\alpha})$ \cite{Giles}, where $\alpha$ is the global discretization bias (i.e. $\abs*{\mathbb{E}[X_{L}]-\mathbb{E}[X]}$) of the numerical method used to simulate $X_{l}$, $l\geq0$. One notes from \cite{Giles} that, 
\begin{equation}
\abs*{\mathbb{E}[X_{l}]-\mathbb{E}[X_{l-1}]} \leq (M-1)ch_{l}^{\alpha}
\label{equation:overallbias}
\end{equation}
where $c$ is a positive constant, and that $(M-1)^{-1}\abs*{\mathbb{E}[X_{L}]-\mathbb{E}[X_{L-1}]}$ can be used as an estimate for the overall discretization bias,  $\abs*{\bar{X}_{L}-\mathbb{E}[X]}$. As each estimator in (\ref{equation:MLestimator}) is independent of one another, the overall variance is given by the sum of the variances of each individual estimator. Given that there is a positive correlation between $X_{l}$ and $X_{l-1}$, one can expect then that the sample variance of $X_{l}-X_{l-1}$, denoted by 
\begin{equation}
V_{l}=\mathbb{V}[X_{l}-X_{l-1}]=\mathbb{V}[X_{l}]+\mathbb{V}[X_{l-1}]-2\mathbb{C}ov[X_{l},X_{l-1}],
\end{equation} 
decays at a rate proportional to $l$, so that $\mathbb{V}_{l}=O(h_{l}^{\beta})$, $\beta>0$. The covariance in the last term in \eqref{equation:overallbias} is taken over the joint probability distribution of $X_{l}$ and $X_{l-1}$. One can then trade off variance in fine/expensive estimators against discretization error in coarse/cheap estimators with lower variance by setting a decreasing sequence of Monte Carlo estimator sample sizes,  $N_{0}>N_{1}...>N_{L}$. The overall computational cost of the MLMC estimator is
\begin{equation}
C_{ML}=\sum^{L}_{l=0}h_{l}^{-\gamma}N_{l}T,
\label{equation:standardMLMCcost}
\end{equation}
where, $h_{l}^{-\gamma}$ defines the computational cost of propagating
one single sample (of $X_{l}$) through a discretized system. On the
other hand, the cost of the standard Monte Carlo estimator in
(\ref{equation:MonteCarlo}) is
\begin{equation}
C_{MC}=h_{L}^{-\gamma}NT.
\label{equation:MonteCarloCost}
\end{equation}
One can choose the continuous variables $N_{l}$, $l=0,...,L$, at a rate that minimises the variance of the MLMC estimator for a fixed computational cost, $N_{l} \propto \sqrt{V_{l}h_{l}^{\gamma}}$. In particular, by following a formula given in \cite{Giles,Cliffe}, one can find optimal values of $N_{l}$, as well as the finest level
$L$, and in doing so achieve a computational cost reduction relative to the standard
Monte Carlo counterpart (\ref{equation:MonteCarlo}), with the same
bound Mean Square error. Giles \cite{Giles} proved the following result.
\begin{theorem}
\label{theorem:MLMC MSE}
If the Mean Square Error of $\bar{X}_{L}$ is bounded by $O(\epsilon^{2})$, one can optimally choose $L$ and $N_{l}$ to allow the computational cost of the MLMC estimator to be bounded by,
\begin{equation}
C_{ML} \leq 
\begin{cases}
c_{1}\epsilon^{-2},& \gamma<\beta \\
c_{2}\epsilon^{-2}log(\epsilon)^{2}, & \gamma = \beta \\
c_{3}\epsilon^{-2-\frac{(\gamma-\beta)}{\alpha}},& \gamma > \beta 
\end{cases}
\end{equation}
where $c_{1}$, $c_{2}$, $c_{3}$, $\gamma$, $\beta$ are positive constants and $\alpha \geq \frac{1}{2}min(\gamma,\beta)$.\\
\end{theorem}
For $\bar{X}^{MC}_{L}$ to have a Mean Square Error of $O(\epsilon^{2})$, a sample size $N$ of $O(\epsilon^{-2})$ is required, as well as a discretization bias given by $h_{L}=O(\epsilon^{\frac{1}{\alpha}})$. Thus any of the computational costs in Theorem \ref{theorem:MLMC MSE} are less than $C_{MC}$ ($O(\epsilon^{-2-\frac{\gamma}{\alpha}})$). The MLMC approach in principle is very simple to implement and can be very effective as long as one can satisfy the two constraints, $\alpha\geq \frac{1}{2}min(\gamma,\beta)$ and $\beta>0$. More detail on this method can be found in a generalised explanation, and related theorems, in \cite{Cliffe}. 

\section{{Particle Filtering}}
\label{sec:PF}

This section will outline the standard particle filtering methodology. In this context, one is interested in computing statistics of a random process $X_{t}$, 
conditioned on observations of a single realisation of $X$, denoted $X'$, and 
referred to as the reference solution. The observations are random variables 
of the form
\begin{equation}
Y_{t_{k}}=H(X_{t_{k}}')+\phi
\end{equation}
at times $t \in (t_1,\ldots,t_{N_{y}})$, where $H:\mathbb{R}^{d} \to
\mathbb{R}^{y}$ is an observation operator, and $\phi$ is a random
variable representing measurement error. For simplicity, we choose
$\phi\sim N(0,R)$, $R$ is a $y \times y$ covariance matrix, and we
will take $H$ to be the identity (so $y$=$d$). We define $X_{L,t_{k}}$ to be a numerical discretization of $X_{t_{k}}$ with discretization accuracy parameter $h_{L}$. Our aim here is to sequentially approximate $\mathbb{E}_{\eta_{L,t_{k}}}[X_{L,t_{k}}]$, the expectation of $X_{L,t_{k}}$ with respect to the measure $\eta_{L,t_{k}}$, where $\eta_{L,t_{k}}$ is the posterior of $X_{L,t}$ given the observations $Y_{t_{1},...,t_{k}}$. Let $p(y|x)$ be the likelihood function of $y$ given $x$ and $q(x)$ be the prior of $x$. Then, for any $k \in [1,N_{y}]$, using Importance Sampling \cite{Doucet}, one can draw $N$ i.i.d samples from the empirical approximation of the prior of $X_{L,t_{k}}$, 
\begin{equation}
\hat{q}(X_{L,t_{k}})=\sum^{N}_{i=1}\tilde{w}_L(X^i_{L,t_{k-1}})\delta(X_{L,t_{k}}-X_{L,t_k}^{i}),
\end{equation}
and denote the normalised importance weights of sample $i \in [1,N]$ to be
\begin{equation}
\tilde{w}_{L}(X_{L,t_k}^{i})=\frac{w_{L}(X_{L,t_k}^{i})}{\sum^{N}_{j=1}w_{L}(X_{L,t_k}^{j} )},
\end{equation}
where
\begin{equation}
w_{L}(X_{L,t_k}^{i}) = p(Y_{t_k}|X_{L,t_k}^{i})\tilde{w}_L(X_{L,t_{k-1}}^{i}).
\label{equation:weights}
\end{equation}
Thus, the filter weights are defined iteratively, starting from $\tilde{w}_{L}(X_{L,t_{0}}^{i})=\frac{1}{N}$.
As the observations are given by a Gaussian distribution, the likelihood is
\begin{equation}
p(Y_{t_{k}}|X_{L,t_{k}}^{i})=\frac{1}{\sqrt{2\pi}\abs*{R}^{y/2}}e^{-\frac{1}{2}\big(H(X_{L,t_{k}}^{i})-Y_{t_{k}}\big)^{T}R^{-1}\big(H(X_{L,t_{k}}^{i})-Y_{t_{k}}\big)}.
\end{equation}
Finally, an estimator for the expectation of $X_{L,t_{k}}$ with respect to the posterior $\eta_{L,t_{k}}$ is
\begin{equation}
\bar{X}_{L,t_{k}}=\sum^{N}_{i=1}\tilde{w}_{L}(X_{L,t_{k}}^{i})X_{L,t_{k}}^{i}.
\label{equation:PosteriorMean}
\end{equation}
This estimator, despite being biased by $O(N^{-1})$ due to the normalised importance weights in a single importance sampling update, is consistent with $\mathbb{E}_{\eta_{L,t_{k}}}[X_{L,t_{k}}]$. This is to say, as $N \to \infty$, the estimator converges in probability to $\mathbb{E}_{\eta_{L,t_{k}}}[X_{L,t_{k}}]$.

Typically, importance weights become degenerate as $k$ increases \cite{Cappe}. In this case, it is necessary to  duplicate higher weighted particles whilst removing lower weighted particles; this is known as resampling. Resampling resembles an unbiased transformation from the weighted ensemble, $\big\{X_{l,t_{k}}^{i},\tilde{w}_{l}(X_{l,t_{k}}^{i})\big\}_{i=1,\ldots,N}$ to an evenly weighted ensemble of resampled particles $\big\{\tilde{X}_{l,t_{k}}^{i}\big\}_{i=1,\ldots,N}$. The scheme outlined above is known as the Sequential Importance Resampling (SIR) method. For more information on this turn to \cite{Doucet}. The Ensemble Transform Particle Filter, the subject of this paper, uses Optimal Transportation \cite{Villani} to implement this transformation, which we describe next.

\subsection{Ensemble Transform Particle Filters}
\label{sec:ETPF}

ETPFs are a variant of linear ensemble transform filters (LETF) \cite{Reich}. They present an alternative to the resampling step that takes place in the standard SIR methodology, replacing it with a linear transformation. The goal is to obtain a transformed set of evenly weighted particles, $\big\{\tilde{X}_{L,t_{k}}^{i}\big\}_{i=1,\ldots,N}$, from the weighted set of particles $\big\{X_{L,t_{k}}^{i}\big\}_{i=1,\ldots,N}$, with importance weights $\big\{\tilde{w}_{L}(X_{L,t_{k}}^{i})\big\}_{i=1,\ldots,N}$, defining an empirical approximation to the posterior distribution $\eta_{L,t_{k}}$. This can be done with the following linear transformation,
\begin{equation}
\tilde{X}_{L,t_{k}}^{j}=\sum^{N}_{i=1}P_{i,j}X_{L,t_{k}}^{i}
\label{equation:LETPF}
\end{equation}
for $i=1,\ldots,N$ and $j=1,\ldots,N$ with non-zero entries for $P_{i,j}$. Here, $\sum^{N}_{i=1}P_{i,j}=1$. Let $Z_{L,t_{k}}$ denote the discrete random variable with samples $X_{L,t_{k}}^{i}$ and associated probability vector $\tilde{w}_{L}(X_{L,t_{k}}^{i})$, $i=1,...,N$. Then take $\tilde{Z}_{L,t_{k}}$ to be the discrete random variable with samples $X_{L,t_{k}}^{i}$, $i=1,...,N$ all with equal probability. For the ETPF, one creates a coupling between $Z_{L,t_{k}}$ and $\tilde{Z}_{L,t_{k}}$, denoted by the matrix $T_{i,j}$, size $N \times N$, with non-negative entries. The coupling defines the linear transformation matrix in (\ref{equation:LETPF}) as $P_{i,j}=NT_{i,j}$. This coupling can be found by solving a linear transport problem 
by minimising the expected Euclidean distance between $Z_{L,t_{k}}$ and $\tilde{Z}_{L,t_{k}}$, subject to the
constraints
\begin{equation}
\sum^{N}_{i=1}T_{i,j}=\frac{1}{N}, \quad
\sum^{N}_{j=1}T_{i,j}=\tilde{w}_{L}(X_{L,t_{k}}^{i}).
\end{equation}
This is in fact equivalent to maximising the covariance between the two ensembles, since
\begin{equation}
\begin{split}
&\mathbb{E}_{Z_{L,t_{k}},\tilde{Z}_{L,t_{k}}}[\norm*{z_{L,t_{k}}-\tilde{z}_{L,t_{k}}}^{2}]=\mathbb{E}_{Z_{L,t_{k}}}[\norm*{z_{L,t_{k}}}^{2}]+\mathbb{E}_{\tilde{Z}_{L,t_{k}}}[\norm*{\tilde{z}_{L,t_{k}}}^{2}] \dots \\
\quad & \dots -2\mathbb{E}_{Z_{L,t_{k}}}[z_{L,t_{k}}]^{T}\mathbb{E}_{\tilde{Z}_{L,t_{k}}}[\tilde{z}_{L,t_{k}}]-2Tr\big(Cov_{Z_{L,t_{k}},\tilde{Z}_{L,t_{k}}}[z_{L,t_{k}},\tilde{z}_{L,t_{k}}]\big).
\end{split}
\label{equation:Identity}
\end{equation}
In a univariate case, we define an optimal coupling matrix, $T_{i,j}$, as one which minimises the cost function,
\begin{equation}
\sum^{N}_{i=1}\sum^{N}_{j=1}T_{i,j}(X_{L,t_{k}}^{i}-X_{L,t_{k}}^{j})^{2}
\end{equation}
Theoretical analysis of the above transformation is given in
\cite{Reich}. Once the transformed particles in
(\ref{equation:LETPF}) are found, the posterior mean is now estimated
by, 
\begin{equation}
\bar{X}_{L,t_{k}}=\frac{1}{N}\sum^{N}_{j=1}\tilde{X}_{L,t_{k}}^{j}.
\end{equation}
It is important to note that this linear transformation, which is
deterministic, will give the same estimator as in
(\ref{equation:PosteriorMean}), and thus does not add considerable
extra variance to the estimator from random resampling.  This is a
consistent estimator for the previous posterior mean estimator
$\Big(\sum^{N}_{i=1}\tilde{w}_{L}(X_{L,t_{k}}^{i})X_{L,t_{k}}^{i}\Big)$,
since
\begin{equation}
\bar{X}_{L,t_{k}}=\frac{1}{N}\sum^{N}_{j=1}\tilde{X}_{L,t_{k}}^{j}=\frac{1}{N}\sum^{N}_{j=1}\sum^{N}_{i=1}T_{i,j}NX_{L,t_{k}}^{i}=\sum^{N}_{j=1}\sum^{N}_{i=1}T_{i,j}X_{L,t_{k}}^{i}=\sum^{N}_{i=1}\tilde{w}_{L}(X_{L,t_{K}}^{i})X_{L,t_{k}}^{i}
\label{equation:ConsistentET}
\end{equation}

In the univariate case, the matrix $T_{i,j}$ can easily be found by an $O\big(Nlog(N)\big)$ algorithm in \cite{EnsembleForecasting}. This will become an important observation when discussing localisation in the next section. The above constraints lead to a maximum of $2N-1$ non-zero elements in $T_{i,j}$, leading to a very sparse matrix calculation, and thus the ensemble transformation process can be achieved in a $O(N)$ computational cost. The $O\big(Nlog(N)\big)$ computational cost comes from the fact that one has to sort the univariate particles prior to the algorithm. In our numerical experiments at the end of this paper, this sorting was a negligible part of the ensemble transform computational cost. This allows one to be able to carry the ensemble transform out on every assimilation step without the computational expense of this being of a higher order of magnitude than the propagation of the particles in between assimilation steps, but more analysis will cover this observation in the next section. \\
In the multivariate case, the same linear transport problem prevails, however one is required to minimise the cost function,
\begin{equation}
\sum^{N}_{i=1}\sum^{N}_{j=1}T_{i,j}\norm*{X_{L,t_{k}}^{i}-X_{L,t_{k}}^{j}}^{2},
\label{equation:multivariatLTP}
\end{equation}
whose minimum defines the Wasserstein distance between $Z_{L,t_{k}}$ and $\tilde{Z}_{L,t_{k}}$. This can be solved in a $O\big(N^{3}log(N)\big)$ computational cost, using algorithms such as the FastEMD algorithm \cite{FastEMD}. However, this means that with many systems, this optimal transportation computational cost will dominate over model costs of the system and thus the scheme is not efficient. The model costs of the particle filter are defined to be the computational cost needed to propagate the $N$ particles through the system in between assimilation steps and is given in (\ref{equation:MonteCarloCost}) (determined by a constant $\gamma$). Thankfully, a technique called \textit{localisation} can aid this problem, and can also provide a pivotal change to the scheme when applying it to high dimensional systems. 

\subsubsection{Localisation}

Localisation, a scheme frequently used in the EnKF for high dimensional systems, can also be applied to the ETPF \cite{Cheng}. In the simplest form, localisation applied to the ETPF means that one can reduce the computational cost of designing a multivariate coupling to $d$ times the cost of designing a univariate coupling. Localisation allows one to construct an individual transformation in (\ref{equation:LETPF}) for each of the $d$ components of a multivariate $X_{L,t_{k}}$. A simple definition for the localisation matrix $C$ \cite{Cheng} that describes the spatial correlation structure of the ensemble $\big\{X_{L,t_{k}}\big\}_{i=1,\ldots,N}$ could be
\begin{equation}
C_{m,n}=
\begin{cases}
1-\frac{1}{2}\big(\frac{s_{m,n}}{r_{loc,c}}\big),& \big(\frac{s_{m,n}}{r_{loc,c}}\big) \leq 2, \\
0,& otherwise.
\end{cases}
\label{equation:localisationmatrix}
\end{equation}
Here $m,n=1,\ldots,d$ are the indicies of the spatial components of $X_{L,t_{k}}$,
\begin{equation}
s_{m,n}=\min\big\{\abs*{m-n-N},\abs*{m-n},\abs*{m-n+N}\big\},
\end{equation}
and $r_{loc,c}$ is a constant. The above form for $C_{m,n}$ explicitly takes spatial-periodicity into account. One can now decompose the linear transport problem in (\ref{equation:multivariatLTP}) into $d$ separate linear transport problems, to find a coupling matrix $T_{i,j}(m)$, $i=1,\ldots,N$, $j=1,\ldots,N$, for each component $m=1,\ldots,d$. The objective of these linear transport problems is minimising the cost function
\begin{equation}
\sum^{N}_{i=1}\sum^{N}_{j=1}T_{i,j}(m)\norm*{X_{L,t_{k}}^{i}-X_{L,t_{k}}^{j}}^{2}_{m},
\end{equation}
where
\begin{equation}
\norm*{X_{L,t_{k}}^{i}-X_{L,t_{k}}^{j}}^{2}_{m}=\sum^{d}_{n=1}C_{m,n}\big(X_{L,t_{k}}^{i}(n)-X_{L,t_{k}}^{j}(n)\big)^{2},
\label{equation:localisationminimizer}
\end{equation}
subject to the constraints
\begin{equation}
\sum^{N}_{i=1}T_{i,j}(m)=\frac{1}{N}, \quad
\sum^{N}_{j=1}T_{i,j}(m)=\tilde{w}_{L}(X_{L,t_{k}}^{i}).
\label{equation:localisationconstraints}
\end{equation}
Here, $X_{L,t_{k}}^{i}(m)$ is the $m$'th component of $X_{L,t_{k}}^{i}$. Then, one can define the approximation of the marginal posterior mean for each $m=1,\ldots,N$ as
\begin{equation}
 \bar{X}_{L,t_{k}}(m)=\frac{1}{N}\sum^{N}_{j=1}\tilde{X}_{L,t_{k}}^{j}(m),
\end{equation}
where the transformed components are given by
\begin{equation}
\tilde{X}_{L,t_{k}}^{j}(m)=\sum^{N}_{i=1}P_{i,j}(m)X_{L,t_{k}}^{i}(m),
\end{equation}
and $P_{i,j}(m)=T_{i,j}(m)N$. Note the cost functions in (\ref{equation:localisationminimizer}) do not achieve the minimum of (\ref{equation:multivariatLTP}). When $r_{loc,c}=0$, exhibiting the most computationally efficient scenario, one has the interesting case where $d$ univariate linear transport problems need to be solved, thus transforming all components individually. One can simply use the univariate algorithm in \cite{EnsembleForecasting}, mentioned in the last section, with a computational cost of $O\big(Nlog(N)\big)$, for each linear transport problem, to get an overall $O\big(dNlog(N)\big)$ computational cost. In practice, when $r_{loc,c}=0$, one can also reorder each of the transformed sets of components into the rank structure of the original ensemble. This preserves the copula structure \cite{Schefzik} of the original ensemble. \\ 
If the model costs of a system are less than that of the multivariate optimal transportation, using $r_{loc,c}=0$ is the only case in which the model costs can return to being the dominative cost in the ETPF estimator. Despite this, it is very reasonable to imagine that this model cost will dominate that of the optimal transportation in some systems, especially for high dimensional Partial Differential Equations (i.e. where $\gamma$ is high). Localisation does have an effect on the performance of the ETPF by adding bias into the posterior mean approximation stemming from the fact that one is generating deterministic couplings $T_{i,j}$ that are minimising different (simplified) cost functions to the full, multivariate one in (\ref{equation:multivariatLTP}). This bias is thus caused by the decay in correlation between the components. Despite this, numerical experiments conducted in \cite{Cheng} find the localised ETPF to be effective even in the chaotic, highly nonlinear Lorenz Equations. \\
Localisation is also needed in the likelihood evaluation of multivariate particles in the ETPF. Although this is not critical to the aim of this paper, only briefly covered here, it is essential for the ETPF to be able to successfully filter high dimensional systems due to the curse of dimensionality. Standard Sequential Monte Carlo (SMC) methods fail to track high dimensional systems due to exponentially degenerate importance weights. However, while there have only been some suggestions for a solution to this problem in SMCs, such as in \cite{Rebeschini}, one can alter the above localisation scheme in the ETPF to solve this problem swiftly \cite{Cheng}. It is also needed to reduce the computational cost of likelihood evaluations when the dimension of the state space is greater than the sample size ($d > N$). For each component ($m=1,\ldots,d$) in the particles $\big\{X_{L,t_{k}}^{i}\big\}_{i=1,\ldots,N}$, generate an separate importance weight given by
\begin{equation}
w_{L}(X_{L,t_{k}}^{i}(m)) \propto \frac{1}{\sqrt{2\pi}\abs*{R}^{y/2}}e^{-\frac{1}{2}\big(H(X_{L,t_{k}}^{i})-Y_{t_{k}}\big)^{T}(\tilde{C}_{m})R^{-1}\big(H(X_{L,t_{k}}^{i})-Y_{t_{k}}\big)},
\label{equation:localisedweight}
\end{equation}
where
\begin{equation}
(\tilde{C}_{m})_{n,n}=
\begin{cases}
1-\frac{1}{2}\big(\frac{s_{m,n}}{r_{loc,R}}\big),& \big(\frac{s_{m,n}}{r_{loc,R}}\big) \leq 2, \\
0,& otherwise,
\end{cases}
\end{equation}
for $n=1, \dots N$ ($\tilde{C}_{m}$ is diagonal) and the value of $r_{loc,R}$ can be independent to $r_{loc,c}$. Of course, $H$ should be a local operator, see \cite{Anderson} for details of the use of localisation within the EnKF. These weights are then used in the constraints in the linear transport problems for each individual component transformation in (\ref{equation:localisationconstraints}). The two `radii' of localisation, $r_{loc,c}$ and $r_{loc,R}$ will henceforth be refered to as the particular settings of localisation used. 

\section{{Multilevel Ensemble Transform Particle Filter (MLETPF)}}
\label{sec:MLETPF}

The proposed multilevel ETPF framework is demonstrated in this section. It creates an estimator consistent with the standard ETPF estimator in (\ref{equation:PosteriorMean}), for the same discretization accuracy level, $L$. The term `single level' estimator will henceforth be a reference to the corresponding standard ETPF estimator, conditioned on the same observations, with the same discretization level $L$ and variance as the proposed MLETPF estimator. The general premise of the MLETPF is to run $L+1$ independent ETPF estimators, with $N_{l}$ samples, forward in time (and space), in the coupled multilevel framework. When updating the weights of each particle in each estimator, the same method as the ETPF holds for each of the $L+1$ estimators. Thus, we define the MLETPF estimator of $\mathbb{E}_{\eta_{L,t_{k}}}[X_{L,t_{k}}]$, as the following, where the importance weights $\tilde{w}_{l}(X_{l,t_{k}})$ target $\eta_{l,t_{k}}$, the posterior of each $d$-dimensional discretization $X_{l,t_{k}}$, given the observations, $Y_{t_{1},...,t_{k}}$, $k \in [1,N_{y}]$,
\begin{equation}
\bar{X}_{L,t_{k}}=\Big(\sum^{N_{0}}_{i=1}\tilde{w}_{0}(X_{0,t_{k}}^{i})X_{0,t_{k}}^{i}\Big)+\sum^{L}_{l=1}\Big(\sum^{N_{l}}_{i=1}\big(\tilde{w}_{l}(X_{l,t_{k}}^{i})X_{l,t_{k}}^{i}-\tilde{w}_{l-1}(X_{l-1,t_{k}}^{i})X_{l-1,t_{k}}^{i}\big)\Big)
\label{equation:MLMCETPF}
\end{equation}
We assume here that $h_{0} \leq \Delta t$, where $\Delta t = t_{k+1} - t_{k}$ for all $k \in [1,N_{y}-1]$, so that all of the $L+1$ estimators are conditioned on the same observations. This does mean that one has to set a bound on the frequency of the data assimilation, given the time-step of the minimum level, $h_{0}$. We note that it is possible to adjust the framework here slightly to incorporate frequent observations only available on finer levels at certain times. This could be done by only updating the weights for finer ensembles at those observations and then proceeding with the ensemble transform stages when both the coarse and fine levels in each difference estimator have had an importance weight update. \\
One notes that as the standard ETPF estimator for each level of descritization, $l \geq 0$, is consistent with $\mathbb{E}_{\eta_{l,t_{k}}}[X_{l,t_{k}}]$, the above estimator is consistent with the $\mathbb{E}_{\eta_{L,t_{k}}}[X_{L,t_{k}}]$, given the linearity of expectation shown in (\ref{equation:linearityexpectation}). Here, each of the particles from the fine and coarse ensembles in each of the multilevel difference estimators are positively correlated in between assimilation steps as in the standard MLMC method. This correlation is required for the variance of each difference estimator to decay with $l \to \infty$ as discussed in the opening section. However, now in the ETPF context, when one comes to transform the fine and coarse ensembles in each multilevel difference estimator, the two ensembles cannot be transformed independently of one another, and need to have a positive correlation imparted between them ready for the next phase of particle propagation, especially if the transformations are happening frequently. If the random input to the system is simply a random initial condition, in a system with no stochastic forcing, these particles from the fine and coarse ensembles will certainly diverge instantly if they are not positively correlated after the ensemble transformations. In this paper, this positive correlation is achieved using a multilevel coupling step after the standard ensemble transform stage. This requires one to first carry out the ensemble transform (\ref{equation:LETPF}) on the coarse and fine ensembles, $\big\{X_{l-1,t_{k}}^{i}\big\}_{i=1,\ldots,N_{l}}$, $\big\{X_{l,t_{k}}^{i}\big\}_{i=1,\ldots,N_{l}}$ with weights $\big\{w_{l-1}(X_{l-1,t_{k}}^{i})\big\}_{i=1,\ldots,N_{l}}$, $\big\{w_{l}(X_{l,t_{k}}^{i})\big\}_{i=1,\ldots,N_{l}}$ respectively, to get evenly weighted particles, $\big\{\tilde{X}_{l-1,t_{k}}^{i}\big\}_{i=1,\ldots,N_{l}}$ and $\big\{\tilde{X}_{l,t_{k}}^{i}\big\}_{i=1,\ldots,N_{l}}$. If localisation is needed, one can implement this with the required parameters on both ensemble transforms as in the last section. It is very important that the same localisation settings are used on all estimators, so that the overall MLETPF estimator is consistent with the single level ETPF estimator with the same localisation settings. The key point of this being that the fine and coarse ensembles from each discretization level will have the same systematic localisation bias as one another. This means, such as with the discretization bias, that the localisation biases can cancel each other out in the telescoping sum of estimators (\ref{equation:MLMCETPF}), leaving only the systematic localisation bias of the finest level equal to that of the localised single level estimator. At this point, one notes that (\ref{equation:MLMCETPF}) becomes
\begin{equation}
\bar{X}_{L,t_{k}}=\Big(\frac{1}{N_{0}}\sum^{N_{0}}_{i=1}\tilde{X}_{0,t_{k}}^{i}\Big)+\sum^{L}_{l=1}\Big(\frac{1}{N_{l}}\sum^{N_{l}}_{i=1}\big(\tilde{X}_{l,t_{k}}^{i}-\tilde{X}_{l-1,t_{k}}^{i}\big)\Big).
\label{equation:TransformedMLMCETPF1}
\end{equation}
Now one needs to positively couple the fine and coarse ensembles of transformed particles from each estimator above. We propose to build another coupling between $\tilde{X}_{l,t_{k}}$ and $\tilde{X}_{l-1,t_{k}}$, denoted by $T^{F/C}_{i,j}$, that minimises the cost function
\begin{equation}
\sum^{N_{l}}_{i=1}\sum^{N_{l}}_{j=1}T^{F/C}_{i,j}\norm*{\tilde{X}_{l,t_{k}}^{i}-\tilde{X}_{l-1,t_{k}}^{j}}^{2},
\label{equation:multilevelcouplingcostfunction}
\end{equation}
with constraints
\begin{equation}
\sum^{N}_{i=1}T^{F/C}_{i,j}=\frac{1}{N_{l}}, \qquad
\sum^{N}_{j=1}T^{F/C}_{i,j}=\frac{1}{N_{l}}.
\label{equation:couplingconstraints}
\end{equation}
This is an assignment problem and in the multivariate case it can be solved by the Hungarian algorithm \cite{Hungarian} with a computational cost equal to the multivariate linear transport problem algorithms discussed previously and so is the same order of magnitude as the corresponding ensemble transform stage in the standard ETPF method. In the univariate case, one can simply use the cheap algorithm in \cite{EnsembleForecasting}, exactly like the ensemble transform stage. One notes that the above assignment problem returns a coupling with one element in each row and column ($\frac{1}{N_{l}}$), resulting in particles simply being reordered and not transformed. This therefore returns exactly the same transformed particles in each ensemble. The reordering can be seen as finding the transformation matrix $P_{i,j}^{F/C}=T^{F/C}_{i,j}N_{l}$ and then applying the standard ensemble transform, in (\ref{equation:LETPF}), to both the fine and coarse transformed ensembles to get new ensembles of  $\big\{\tilde{\tilde{X}}_{l,t_{k}}^{i}\big\}_{i=1,\ldots,N_{l}}$ and $\big\{\tilde{\tilde{X}}_{l-1,t_{k}}^{i}\big\}_{i=1,\ldots,N_{l}}$ which are now positively correlated. Each multilevel difference estimator can now be estimated by
\begin{equation}
\hat{X}_{l,t_{k}}=\frac{1}{N_{l}}\sum^{N_{l}}_{j=1}\big(\tilde{\tilde{X}}_{l,t_{k}}^{j}-\tilde{\tilde{X}}_{l-1,t_{k}}^{j}\big).
\label{equation:finerposteriormean}
\end{equation}
Using a calculation similar to \eqref{equation:ConsistentET},
we now show that the estimator in (\ref{equation:finerposteriormean}) is consistent with the term\begin{equation} 
\Big(\sum^{N_{l}}_{i=1}\big(\tilde{w}_{l}(X_{l,t_{k}}^{i})X_{l,t_{k}}^{i}-\tilde{w}_{l-1}(X_{l-1,t_{k}}^{i})X_{l-1,t_{k}}^{i}\big)\Big)
\end{equation}
 from Equation (\ref{equation:MLMCETPF}). Let $T^{F}_{i,j}$ and $T^{C}_{i,j}$ ($i,j=1,\ldots,N$) be the coupling matrices used for the ensemble transform on the finer and coarse ensembles respectively, then
\begin{equation}
\begin{split}
&\hat{X}_{l,t_{k}}
=\frac{1}{N_{l}}\sum^{N_{l}}_{j=1}\big(\tilde{\tilde{X}}_{l,t_{k}}^{j}-\tilde{\tilde{X}}_{l-1,t_{k}}^{j}\big)
=\sum^{N_{l}}_{j=1}\sum^{N_{l}}_{i=1}T^{F/C}_{i,j}\big(\tilde{X}_{l,t_{k}}^{i}-\tilde{X}_{l-1,t_{k}}^{i}\big)\\
\quad &=\frac{1}{N_{l}}\sum^{N_{l}}_{i=1}\sum^{N_{l}}_{j=1}N_{l}\big(T^{F}_{i,j}X_{l,t_{k}}^{j}-T^{C}_{i,j}X_{l-1,t_{k}}^{j}\big)
=\sum^{N_{l}}_{j=1}\big(\tilde{w}_{l}(X_{l,t_{k}}^{j})X_{l,t_{k}}^{j}-\tilde{w}_{l-1}(X_{l-1,t_{k}}^{j})X_{l-1,t_{k}}^{j}\big).
\end{split}
\label{equation:MLETPFconsistency}
\end{equation}
The estimator in (\ref{equation:TransformedMLMCETPF1}) can therefore be written as
\begin{equation}
\bar{X}_{L,t_{k}}=\Big(\frac{1}{N_{0}}\sum^{N_{0}}_{i=1}\tilde{\tilde{X}}_{0,t_{k}}^{i}\Big)+\sum^{L}_{l=1}\Big(\frac{1}{N_{l}}\sum^{N_{l}}_{i=1}\big(\tilde{\tilde{X}}_{l,t_{k}}^{i}-\tilde{\tilde{X}}_{l-1,t_{k}}^{i}\big)\Big),
\label{equation:TransformedMLMCETPF}
\end{equation}
with covariance
\begin{equation}
\mathbb{C}ov[\bar{X}_{L,t_{k}}]=\sum^{L}_{l=0}\frac{\mathbb{V}_{l}}{N_{l}},
\end{equation}
where
\begin{equation}
\mathbb{V}_{l}=
\begin{cases}
\mathbb{C}ov[\tilde{\tilde{X}}_{0,t_{k}}],& l=0, \\
\mathbb{C}ov[\tilde{\tilde{X}}_{l,t_{k}}-\tilde{\tilde{X}}_{l-1,t_{k}}],& l>0, 
\end{cases}
\end{equation}
due to the fact that each estimator in (\ref{equation:TransformedMLMCETPF}) is independent. The above estimator is consistent with the single level (localised, with the same settings as used in the MLETPF) ETPF estimator due to (\ref{equation:MLETPFconsistency}) and (\ref{equation:MLMCETPF}). Therefore, in the absence of  localisation, it is also consistent with $\mathbb{E}_{\eta_{L,t_{k}}}[X_{L,t_{k}}]$. 

The multilevel coupling $T^{F/C}_{i,j}$ minimises the expected
distance between the two transformed ensembles, which maximises the
covariance between them via (\ref{equation:Identity}) and then finally
minimises $V_{l}$. The multilevel coupling procedure above minimises
$\mathbb{V}_{l}$ in each multilevel difference estimator; we also
ensure that pairs of particles in the coarse and fine ensembles are
positively coupled in between assimilation steps (by using the same
random input). The aim of this is to make the covariance
$\mathbb{V}_{l}$ decrease at an asymptotic rate $O(h_{l}^\beta)$
($\beta>0$) required for the variance reduction of the multilevel
framework to work. This is because the fine and coarse ensembles are
coupled both in between assimilation steps and during them via the
coupling. This will be demonstrated in numerical experiments later in
the paper. Designing this coupling between the both transformed
ensembles $\tilde{X}_{l-1,t_{k}}$ and $\tilde{X}_{l,t_{k}}$ is the key
to the proposed MLETPF method, and enforces correlation amongst both
transformed ensembles whilst remaining consistent with the single
level ETPF estimator. Most importantly it also suits the ETPF method
since the coupling can be generated simply and cheaply when using
the $r_{loc,c}=0$ localisation that can be used freely in the
ETPF. This will be shown later.

\subsection{Algorithm}

\label{subsection:algorithm}

In this section, we present an algorithm to implement the MLETPF in practice. In this paper, a pre-defined recurrence relation for the decay of $N_{l}$ as $l \to \infty$ will be set ($N_{l+1}=f(N_{l})$) and these sample sizes will be kept fixed throughout the filtering process. The finest discretization level $L>0$ will be kept arbitrary for now. The algorithm is now presented.

\begin{remunerate}

\item Start at $t_{0}$ (thus $k=0$) and with $l=0$. Choose $N_{0}$.

\item Calculate $N_{l}=\ceil*{f(N_{l-1})}$ if  $l>0$. Sample $\big\{ X_{0,t_{0}}^{i} \big\}_{i=1,\ldots,N_{0}} \sim \pi^{0}$ if $l=0$ or $\big\{ X_{l,t_{0}}^{i}, X_{l-1,t_{0}}^{i} \big\}_{i=1,\ldots,N_{l}} \sim \pi^{0}$, where $X_{l,t_{0}}^{i}=X_{l-1,t_{0}}^{i}$ if $l>0$, at time $t_{0}$. 

\item Propagate all samples forward according to system dynamics until time $t_{k+1}$. If $l>0$, the fine and coarse pairs of samples in each estimator must be coupled by using the same random input.

\item Derive the normalised importance weights for $X_{l,t_{k+1}}$:

\begin{equation}
\tilde{w}_{l}(X_{l,t_{k+1}}^{i})=\frac{w_{l}(X_{l,t_{k+1}}^{i})}{\sum^{N_{l}}_{j=1}w_{l}(X_{l,t_{k+1}}^{j})}
\end{equation}

\item Transform the ensembles, $\big\{X_{l,t_{k}}^{i},\tilde{w}_{l}(X_{l,t_{k}}^{i})\big\}_{i=1,\ldots,N_{l}}$ and $\big\{X_{l-1,t_{k}}^{i},\tilde{w}_{l-1}(X_{l-1,t_{k}}^{i})\big\}_{i=1,\ldots,N_{l}}$ into the evenly weighted ensembles, $\big\{\tilde{X}_{l,t_{k}}^{i}\big\}_{i=1,\ldots,N_{l}}$ and $\big\{\tilde{X}_{l-1,t_{k}}^{i}\big\}_{i=1,\ldots,N_{l}}$, using the linear transformation in (\ref{equation:LETPF}).
Couple them using the multilevel coupling matrix $T^{F/C}_{i,j}$ to produce reordered ensembles $\big\{\tilde{\tilde{X}}_{l,t_{k}}^{i}\big\}_{i=1,\ldots,N_{l}}$ and $\big\{\tilde{\tilde{X}}_{l-1,t_{k}}^{i}\big\}_{i=1,\ldots,N_{l}}$.

\item Move on to Step 7 if $l=L$ or if not, iterate $l+=1$ then repeat steps 2-6 for $k=0$ or steps 3-6 for $k>0$.

\item Iterate k+=1. Start again from step 3 with $l=0$. The MLETPF approximation of $\mathbb{E}_{\eta_{L,t_{k}}}[X_{L,t_{k}}]$ is given by (\ref{equation:TransformedMLMCETPF}).

\end{remunerate}

\subsection{Computational Cost of the MLETPF and ETPF}

As previously noted, in a multidimensional case, it is 
computationally expensive to generate the coupling matrices needed to
couple the fine and coarse transformed particles.
Localisation is used to reduce the computational cost of the ensemble transforms down to $O\big(dNlog(N)\big)$ when $r_{loc,c}=0$ in the standard and multilevel ETPF methods, and this too can also reduce the computational cost of generating the multilevel coupling $T_{i,j}^{F/C}$. When localisation is used along with $r_{loc,c}=0$, one can break the multivariate coupling $T_{i,j}^{F/C}$ down into $d$ separate univariate couplings. As the components are transformed individually, one can simply find a coupling $T^{F/C}_{i,j}(m)$ for each individual component $m$ in the transformed coarse and fine ensembles with the cost function,
\begin{equation}
\sum^{N_{l}}_{i=1}\sum^{N_{l}}_{j=1}T^{F/C}_{i,j}(m)\big(\tilde{X}_{l,t_{k}}^{i}(m)-\tilde{X}_{l-1,t_{k}}^{j}(m))\big)^{2},
\end{equation}
and the same constraints as in (\ref{equation:couplingconstraints}). Each of these couplings can again be found  using the cheap, $O\big(N_{l}log(N_{l})\big)$ (for each $l$) univariate algorithm in \cite{EnsembleForecasting}. One can then reorder / transform each component of fine and coarse ensembles separately using the same methodology as in the last section. This performs a similar role as 
 resampling $N_{l}$ particles from $F^{-1}_{f,m}(u)$ and $F^{-1}_{c,m}(u)$ where $F^{-1}_{f,m}$ / $F^{-1}_{c,m}$ are the marginal (empirical) inverse cumulative distribution functions of $\tilde{X}_{l,t_{k}}(m)$ and $\tilde{X}_{l-1,t_{k}}(m)$ respectively. Here the same uniform variate $u \in [0,1]$ is used for each pair of the $N_{l}$ samples. If localisation is carried out along with $r_{loc,c}>0$, the computational cost of the optimal transport in the standard ETPF and thus the multilevel coupling, minimizing the full multivariate cost function in (\ref{equation:multilevelcouplingcostfunction}), in the MLETPF will rise to $O\big(dN_{l}^{3}log(N_{l})\big)$. This is because $d$ different localised, but still multivariate, optimal transportation problems will have to be solved. Therefore, in this scenario, the model costs are likely to be dwarfed by these optimal transportation costs, which are fixed by definition. In this case there is no justification for implementing the multilevel framework as it only aims to reduce model cost and not the optimal transportation computational expense. However to consider the case where model cost does dominate that of the multivariate optimal transportation, the full multivariate coupling will be demonstrated in the numerical exeriments at the conclusion of this paper. \\
This paper now considers the overall computational cost of the ETPF and MLETPF estimators when localised, with $r_{loc,c}=0$. In this case, as explained above, one can reduce the computational expense of not only the ensemble transform stage, but the multilevel coupling stage as well, in the MLETPF scheme from a potential $O\big(N_{l}^{3}log(N_{l})\big)$ to $O\big(dN_{l}log(N_{l})\big)$ for each multilevel difference estimator. This is enough, with suitable assumptions, to expect that the model computational cost bounds for the standard MLMC method in Theorem \ref{theorem:MLMC MSE} are of the same order of magnitude to that of the entire MLETPF, including the ensemble transform and coupling stages, as one can simply `hide' the optimal transportation costs behind the particle propagation costs. This follows from the proposition.
\begin{proposition}
If $\Delta t \geq h_{0}$ is constant , and one can bound the computational cost of all $N_{y}$ ensemble transform / multilevel coupling stages (with the last term being the cost associated to the sort prior to the algorithm) of the MLETPF by,
\begin{equation}
C_{ET} \leq \sum^{L}_{l=0}\Big(dc_{1}N_{l}N_{y} + de_{1}N_{l}log(N_{l})N_{y} \Big)
\end{equation}
where $c_{1}$ is a positive constant, then the total computational cost of the MLETPF is bounded by,
\begin{equation}
C_{MLETPF} \leq \sum^{L}_{l=0}\Big(c_{2}N_{l}C_{l,d}t_{N_{y}} + de_{1}N_{l}log(N_{l})N_{y}\Big)
\end{equation}
where $C_{l,d}$ is the cost of propagation of one particle on level $l$ (dependent on $d$), $e_{1}$ is a positive constant and $c_{2}$ is a positive constant. 
\label{claim:cost}
\end{proposition}

\begin{proof}
Let the total computational cost of the MLETPF be given by the sum of the model cost and the ensemble transform cost (including the localised likelihood evaluation in (\ref{equation:localisedweight}), which scales at $O(N_{l})$ due to the sparse diagonal $d \times d$ matrix $\tilde{C}$, with $r_{loc,R}$ assumed constant ($r_{loc,R}=O(1)$) and $<<d$)
\begin{equation}
C_{MLETPF}=C_{ET}+C_{MODEL}
\end{equation}
then bounding $C_{ET}$ as in the claim and using the model cost in (\ref{equation:standardMLMCcost}),
\begin{equation}
\begin{split}
C_{MLETPF} & \leq \sum^{L}_{l=0}\Big(N_{l}(dc_{1}N_{y}+C_{l,d}c_{3}t_{N_{y}}) + de_{1}N_{l}log(N_{l})N_{y}\Big)\\
\quad & \leq \sum^{L}_{l=0}\Big(N_{l}(dc_{4}t_{N_{y}}+C_{l,d}c_{3}t_{N_{y}}) + de_{1}N_{l}log(N_{l})N_{y}\Big)\\
\quad & \leq \sum^{L}_{l=0}\Big(c_{2}N_{l}C_{l,d}t_{N_{y}} + de_{1}N_{l}log(N_{l})N_{y}\Big)
\end{split}
\end{equation}
Here we assume that $C_{l,d}$ will grow at least linearly with $d$, $c_{4}=\frac{c_{1}}{\Delta t}$ and that $c_{3}$ is a positive constant. 
\qquad
\end{proof}
\newline The last term of each of the expressions above comes from the sorting of the univariate particles in the cheap algorithm. Although this computational cost scales at $O\big(N_{l}log(N_{l})\big)$, the constant $e_{1}$ is typically very small with respect to the other constants in the cost expression; furthermore this scaling is a worst case estimate, in many cases this would not be reached. Therefore assuming that this part is less than the remainder of the localised ensemble transform cost, one can bound the overall cost of the MLETPF by the model cost, and keep the corresponding reductions outlined in Theorem \ref{theorem:MLMC MSE}. However, even if this sorting cost does dominate the ensemble transform cost, one still expects to recover computational cost reductions relative to the single level estimator, for a fixed error. Despite the model cost in this case not neccessarily being the highest computational expense in the localised MLETPF, with the cost of the sorting / ensemble transform dominating, it is important to note that for a fixed error bound, this sorting cost / ensemble transform will still typically be less than the model computational cost of the single level ETPF. The reductions in Theorem \ref{theorem:MLMC MSE} would then be a slight underestimation however there would still be evident reductions of computational cost relative to the single level ETPF. This is indeed the case for the numerical examples in the next section, as we show there.

\subsection{{Numerical Examples}}

Numerical examples of the MLETPF method, applied to classical data assimilation problems, are given in this section. Three problems will be studied: the multivariate, chaotic Stochastic Lorenz-63 Equations, the univariate, but nonlinear double-well OU Process and the high dimensional Stochastic Lorenz-96 Equations. 
The algorithm above will be used to generate experimental MLETPF estimators (that are compared against the single level ETPF estimators) for the latter two of the three problems above, along with varying levels of pre-defined of accuracy. This pre-defined level of accuracy, $O(\epsilon)$, will determine $L$ (the finest level / overall numerical discretization bias), the fixed sample sizes for each estimator in the MLETPF method ($N_{l}$) and the fixed overall sample size in the corresponding single level ETPF estimator required to achieve the order of magnitude of this error in both estimators. This follows from the standard Monte Carlo approximation error decomposition given by the Central Limit Theorem, as in \cite{Giles}. One can then compare the computational cost, given as the number of operations, for both the single level and multilevel estimators, which should be in line with Theorem \ref{theorem:MLMC MSE} given the same pre-defined error. The error is not bounded exactly due to variations in the variance at each assimilation step, but a proof of concept from a practioner's viewpoint can be established. \\
The error in the estimators will be estimated by the time-averaged root mean square error (RMSE), given by
\begin{equation}
RMSE=\sqrt{\frac{1}{N_{y}}\sum^{N_{y}}_{k=1}\norm*{\bar{X}_{L,t_{k}}-\mathbb{E}_{\eta_{t_{k}}}[X_{t_{k}}]}^{2}},
\end{equation}
where $\eta_{t_{k}}$ is the posterior distribution of $X_{t_{k}}$ given the observations $Y_{t_{1},...,t_{k}}$. An approximation of $\mathbb{E}_{\eta_{t_{k}}}[X_{t_{k}}]$ will be used in the RMSE calculations above, by computing a standard ETPF estimator for it, of which the numerical discretization bias and sample size produce an estimator with an error orders less than any $\epsilon$ used in the following experiments. Where localisation is used in the single level and multilevel estimators for the problems above, the estimators are inconsistent with $\mathbb{E}_{\eta_{L,t_{k}}}[X_{L,t_{k}}]$, but crucially consistent with each other as the same localisation settings are used. The ETPF approximation of $\mathbb{E}_{\eta_{t_{k}}}[X_{t_{k}}]$ will use the same localisation settings to correctly compare the ETPF and MLETPF estimators like for like. 

The recurrence relation for $N_{l}$, $l \geq 1$, given $N_{0}$ (dependent on $\epsilon$), used in both numerical experiments will be set to $N_{l+1}=\ceil*{N_{l}M^{-3/2}}$, however the optimality of this depends on the relative value of $\beta$ with respect to $\gamma$ \cite{Giles}. This is only optimal for $\beta>\gamma$, $(\beta+\gamma)=3$; this is the case in the numerical examples below. One notes that for a RMSE of $O(\epsilon)$, one requires that $N_{l}=O(\epsilon^{-2})$, and that $N=O(\epsilon^{-2})$ for the single level ETPF. Also, the discretization bias of the multilevel and single level ETPF estimators, from $\mathbb{E}_{\eta_{t_{k}}}[X_{t_{k}}]$, should be  $O(\epsilon)$. Thus, for a numerical scheme that has a global discretization bias of $O(t_{N_{y}}h_{L}^{\alpha})$, one requires
\begin{equation}
L = \ceil*{\frac{log\big((t_{N_{y}}^{\alpha}d)/\epsilon\big)}{\alpha log(M)}}
\end{equation}
for the sum of the multilevel and single level ETPF estimator components bias to be of $O(\epsilon)$. One could also use the maximum among all components' biases to be a suitable measure here. The sample covariances of the independent ETPF estimators will also be measured by a sum among all components: $Tr(\mathbb{V}_{l})$, the trace of the covariance matrix. \\
For the single level ETPF, using the above analysis, one can see how requiring $h_{L}=O(\epsilon)$ and $N=O(\epsilon^{-2})$, the model cost in Equation (\ref{equation:MonteCarloCost}) will be $O(\epsilon^{-3})$, dominating the localised (with $r_{loc,c}=0$) ensemble transform cost $O\big(dNlog(N)\big)$, which is equivalent to $O\big(d\epsilon^{-2}log(\epsilon^{-2})\big)$; this supports the point made in the last section. This is also greater than the localised ensemble transform cost of the MLETPF, $O\big(dN_{l}log(N_{l})\big)$, again equivalent to $O\big(d\epsilon^{-2}log(\epsilon^{-2})\big)$. Thus the computational cost reductions of the MLETPF, even if the sorting cost of the localised ensemble transform dominates, is apparent in these cases. The computational cost for the these numerical examples is defined as the theoretical number of operations needed to compute each approximation, including the ensemble transform stage and the multilevel coupling for the MLETPF. This is computed simply by inserting a step-counter into the numerical implementation (functions etc.) in the Python code used for these experiments. Finally, $M=2$ is used for each numerical example.

\subsubsection{Stochastic Lorenz-63 Equations}

This simple 3-component chaotic nonlinear system in $X=(x,y,z)$,
\begin{equation}
\frac{dX}{dt}=
\begin{cases}
\sigma(y-x)+\phi\frac{dW}{dt},& \\
x(\rho-z)-y+\phi\frac{dW}{dt},& \\
xy-\beta z+\phi\frac{dW}{dt},&
\end{cases}
\end{equation}
with $\rho=28$, $\sigma=10$, $\beta=8/3$, $\phi=0.4$ will be used to demonstrate the effect of the multivariate multilevel coupling, $T^{F/C}_{i,j}$, without localisation and thus with cost function in (\ref{equation:multilevelcouplingcostfunction}). The localised coupling will also be used as a comparison. Here, the Brownian motion $W$ will be the same for each component to keep the strong nonlinearity in the equations. Computational cost against accuracy comparisons with the standard ETPF method will not be investigated here given the low model cost of this test problem and thus the dominating effects of the multilevel coupling and / or the ensemble transform stage in both the ETPF and the MLETPF will make the model cost reductions of the multilevel framework unnoticeable. The MLETPF estimator with $N_{0}=500$, and $L=5$, using the full multivariate cost function in (\ref{equation:multilevelcouplingcostfunction}) to find the multilevel couplings $T^{F/C}_{i,j}$ in the aforementioned algorithm, is used to compute an approximation to $\mathbb{E}_{\eta_{L,t_{k}}}[X_{L,t_{k}}]$, with $X$ as above, for $k \in [1,5120]$, with $h_{0}=M^{-7}=\Delta t$, and thus $t_{N_{y}}=40$. Here, $X_{L}$ is the solution to the above Lorenz-63 equations using the forward Euler numerical scheme. The reason that we choose the minimum level to be equivalent to $l=7$ is for stability when using the Euler method. Using different numerical methods, with greater stability at greater time-steps, and thus lower levels would be able to decrease this minimum level. The observations are given by a measurement error with $R=2I$, where $I$ is the $3 \times 3$ identity matrix and weights are thus based on observations of all components $x$, $y$ and $z$. Figure \ref{figure:Experiment0_MeanAsymptotic} shows the mean estimates of $Tr(\mathbb{V}_{l})$ ($l \in [1,5]$) over all assimilation steps $k \in [1,N_{y}]$. The asymptotic decay of the above estimates show importantly that the multilevel coupling, with the multivariate cost function, successfully produces the variance decay, $Tr(\mathbb{V}_{l})=O(h_{l}^{\beta})$, in this case $\beta \approx 1$. The figure also shows the value of $\beta$ (variance decay) for the case where $r_{loc,c}=0$ localisation for the coupling. Localisation being used in a problem such as the strongly nonlinear Lorenz 63 system is dangerous due to the decay in correlations between components, but with the parameters above, it is used simply to compare the rates of variance decay with the non-localised case. One recovers $\beta \approx 2$ in the localised case; this showing that refining the optimal transport down to the one dimensional localisation case is beneficial for variance reduction but comes with the sacrifice of inconsistency of the reference ``single level'' estimator .

\begin{figure}[ht]
  \centering
  \includegraphics[width=100mm]{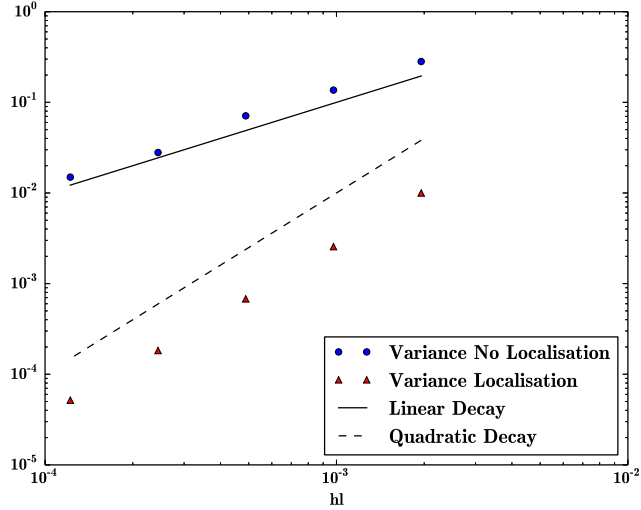}
  \caption{\textit{Mean, over all assimilation steps, $k \in [1,N_{y}]$, of the estimates of $Tr(\mathbb{V}_{l})$, with $l \in [1,5]$ for the Stochastic Lorenz-63 Equations. Both the non-localised and the $r_{loc,c}=0$ localised cases are shown.}}
  \label{figure:Experiment0_MeanAsymptotic}
\end{figure}

\subsubsection{Double-well OU Process}

This is a univariate test problem that demonstrates the cost effective, consistent, MLETPF estimator of $\mathbb{E}_{\eta_{L,t_{k}}}[X_{L,t_{k}}]$, where $X_{L,t_{k}}$ is a numerical discretized solution to the double-well, nonlinear OU process,
\begin{equation}
dX_{t_{k}}=-V'(X_{t_{k}})dt+\xi dW_{t_{k}}
\end{equation}
$k \in [1,N_{y}]$ and $W_{t_{k}}$ is a standard Brownian Motion. Here, $V(X_{t_{k}})=\frac{1}{4}X_{t_{k}}^{4}-\frac{1}{2}X_{t_{k}}^{2}$. This example uses $h_{l}=2^{-4-l}$ but this is arbitrarily chosen. The stochastic forcing is set to $\xi=0.5$. The observations and assimilation times were given by $R=0.6$, $t_{N_{y}}=50$ where $\Delta t=h_{0}$ and so $N_{y}=800$. The numerical discretizations of $X_{t_{k}}$ are computed by the Euler-Maruyama numerical scheme. The parameters above produce a stable numerical solution for a single realisation of the above system when using this scheme for the time frame above. A very accurate simulation ($N_{0}=10000$, $L=7$) of the MLETPF estimator is run to demonstrate the mean asymptotic decay of $V_{l}$ and $\abs*{\hat{X}_{l,t_{k}}}$ ($l \in [1,7]$) over all assimilation steps. These are shown in Figure \ref{figure:Experiment1_MeanAsymptotic}. The values of $\alpha \approx 1$, $\beta \approx 2$, are as expected given the Euler-Maruyama global discretization bias of $O(h_{l})$ and the additive noise in the OU Process contributing to the variance.

\begin{figure}[ht]
  \centering
  \includegraphics[width=100mm]{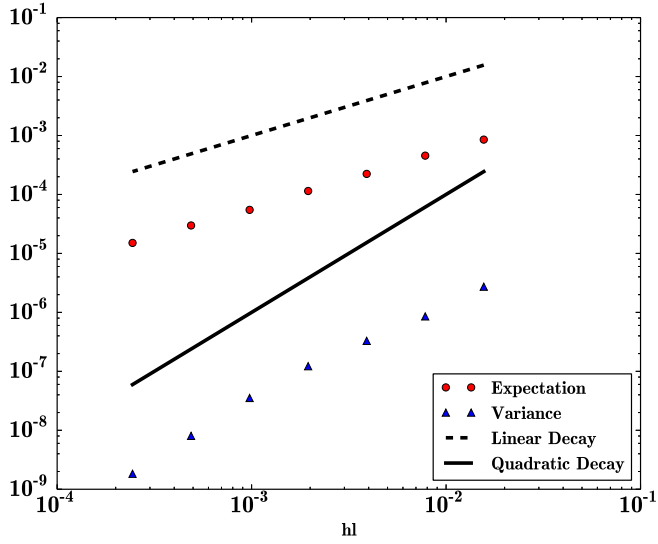}
  \caption{\textit{Mean, over all assimilation steps, $k \in [1,N_{y}]$, of the estimates of $\mathbb{V}_{l}$ (variance) and $\abs*{\hat{X}_{l,t_{k}}}$ (expectation) with $l \in [1,7]$ for the double-well OU process.}}
  \label{figure:Experiment1_MeanAsymptotic}
\end{figure}

\begin{figure}[ht]
  \centering
  \includegraphics[width=100mm]{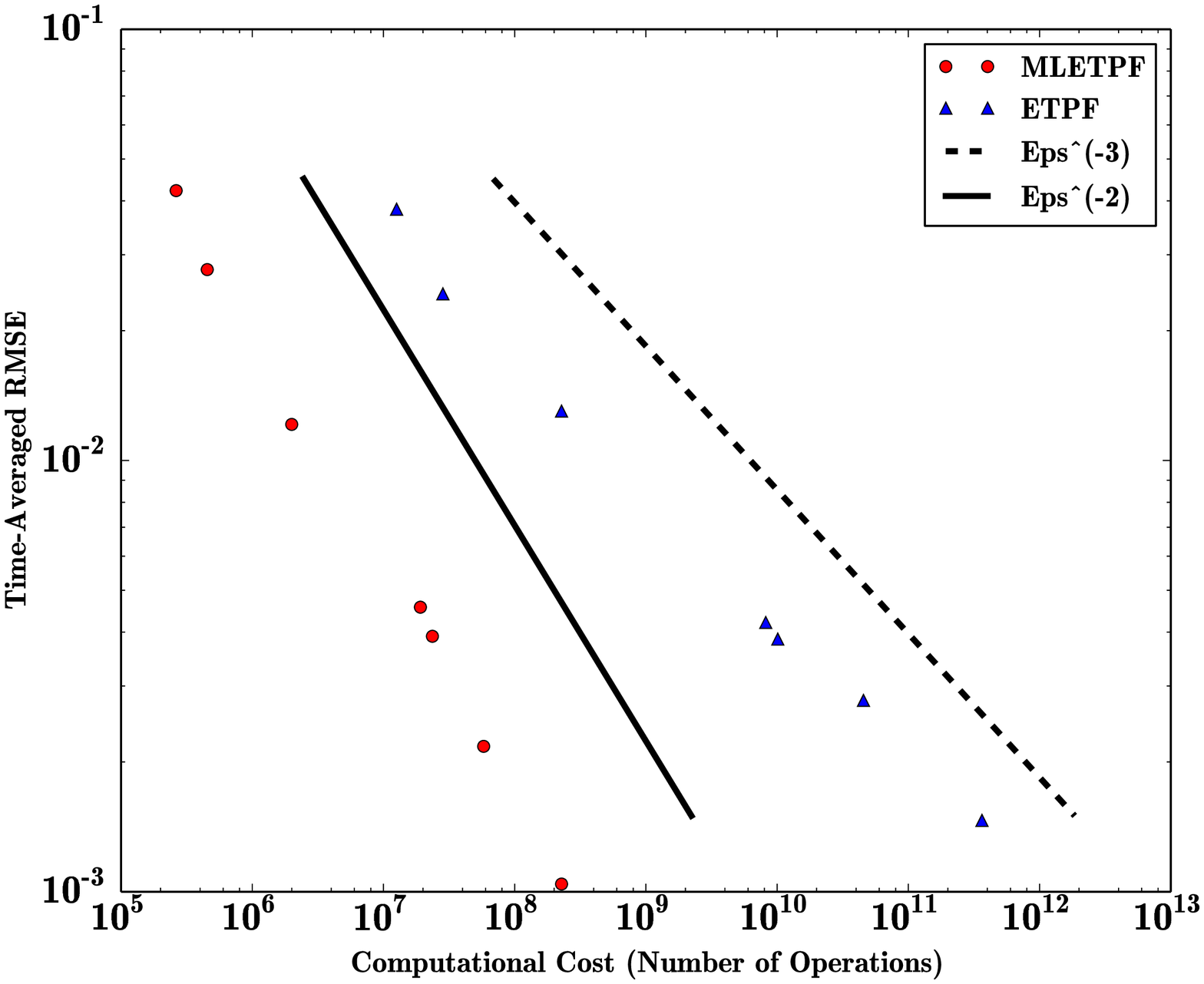}
  \caption{\textit{Computational Cost (number of operations) against the time-averaged RMSE of the ETPF and MLETPF estimators for the double-well OU process. Reference lines show the orders of decay of $RMSE^{-2}$ and $RMSE^{-3}$.}}
  \label{figure:Experiment1_CostVsAccuracy}
\end{figure}

Figure \ref{figure:Experiment1_CostVsAccuracy} shows the computational cost against the accuracy (RMSE) for the MLETPF and the single level ETPF estimators over varying values of $\epsilon$. Here one sets $N_{0}=\epsilon^{-2}$ for the MLETPF and $N=\epsilon^{-2}$ for the single level ETPF estimator. One can clearly see the expected orders of growth for the computational cost of the standard ETPF ($O(\epsilon^{-3})$, as $\gamma=1$ and $\alpha=1$) and the MLETPF ($O(\epsilon^{-2})$, given that $\gamma<\beta$) that were shown in Theorem \ref{theorem:MLMC MSE} for the pre-defined RMSE of $O(\epsilon)$.

\subsubsection{The Stochastic Lorenz-96 Equations}

The final numerical test for the MLETPF method in this paper is the high dimensional ($d=40$ in this case) Stochastic Lorenz-96 Equations, given by,
\begin{equation}
\frac{dX_{j}}{dt}=-\frac{\big(X(j-1)X(j+1)-X(j-2)X(j-1)\big)}{3\Delta X}-X(j)+F+\sigma^2\frac{dW_j}{dt},
\end{equation}
with $j \in [1,N_{x}]$, $N_{x}=40$, $N_{x}\Delta x=10$ and thus $\Delta X=0.25$;
$dW_j/dt$ are 40 i.i.d. Brownian motions.
 Here, $F$ is a constant forcing ($F=8$) and $\sigma^2=0.4$. Periodic boundary conditions are used, so that $X(-1)=X(N_{x})$. The observations were given by a measurement error of $R=6I$, where $I$ is the $40 \times 40$ identity matrix and assimilation times were set to $t_{N_{y}}=100$ where $\Delta t=h_{0}$ meaning $N_{y}=1600$ as $h_{l}=2^{-4-l}$ (again simply arbitrary). The Euler-Maruyama method is used to find $X_{l,t_{k}}$ once again here. In this numerical example, the ETPF and MLETPF estimators use $r_{loc,R}=1$. The localisation setting of $r_{loc,c}=0$ is used for both the multilevel and single level ETPF estimators here, due to the model cost, $C_{l,d}$, being simply equal to $h_{l}^{-1}d$, and thus much lower than that of high optimal transportation costs in multiple dimensions. \\
Once again, a very accurate simulation ($N_{0}=1000$, $L=10$) of the MLETPF was generated to demonstrate the mean asymptotic decays of $\sum^{40}_{i=1}\abs*{\hat{X}_{l,t_{k}}(i)}$ and $Tr(\mathbb{V}_{l})$ ($l \in [1,10]$) over all assimilation steps and these are shown in Figure \ref{figure:Experiment2_MeanAsymptotic}. These follow the same, expected, values of $\alpha \approx 1$, $\beta \approx 2$, as with the last example. 

\begin{figure}[ht]
  \centering
  \includegraphics[width=100mm]{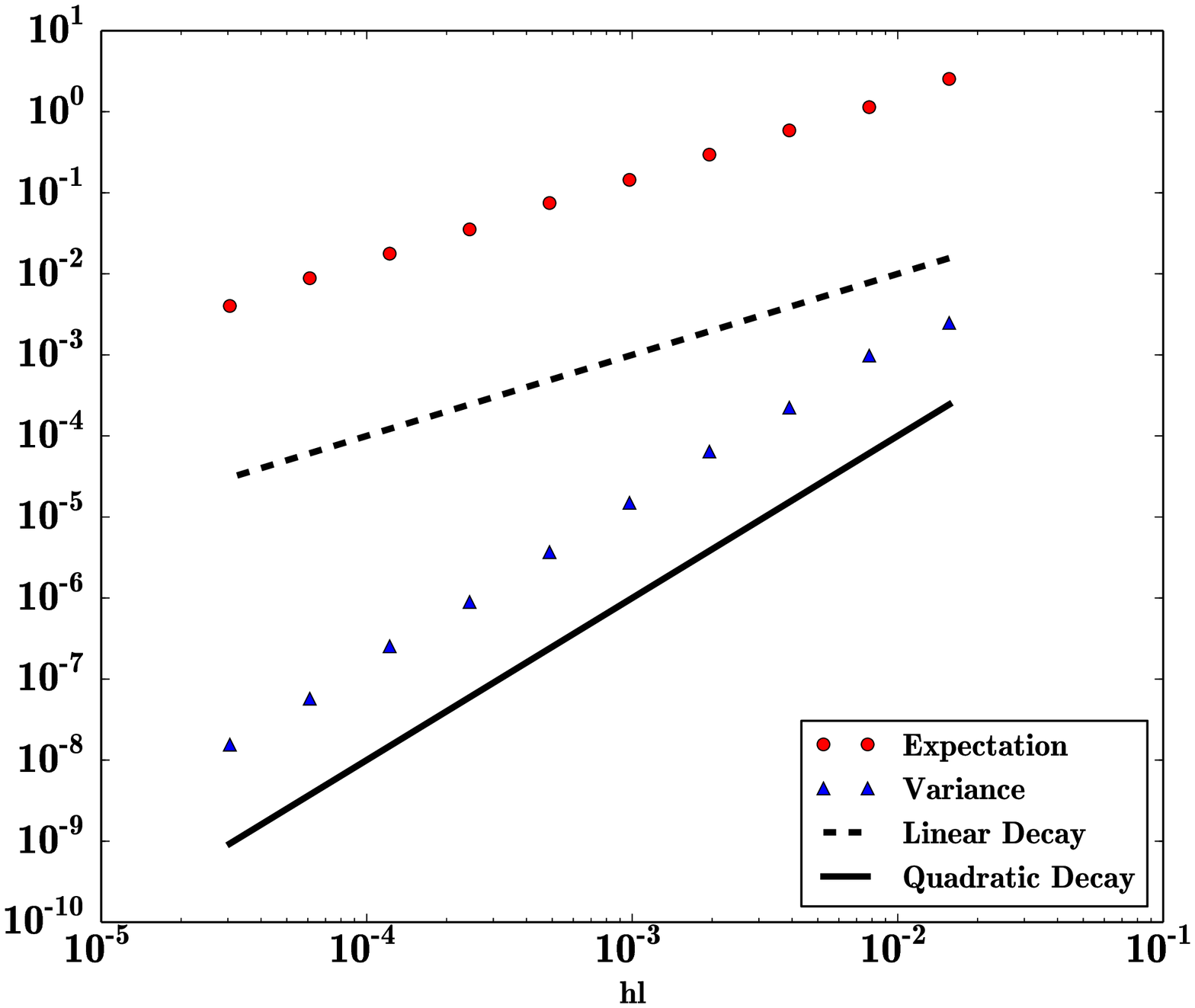}
  \caption{\textit{Mean, over all assimilation steps, $k \in [1,N_{y}]$, of the estimates of $Tr(\mathbb{V}_{l})$ (variance) and $\sum^{40}_{i=1}\abs*{\hat{X}_{l,t_{k}}(i)}$ (expectation) with $l \in [1,10]$ for the Stochastic Lorenz-96 Equations.}}
  \label{figure:Experiment2_MeanAsymptotic}
\end{figure}

\begin{figure}[ht]
  \centering
  \includegraphics[width=100mm]{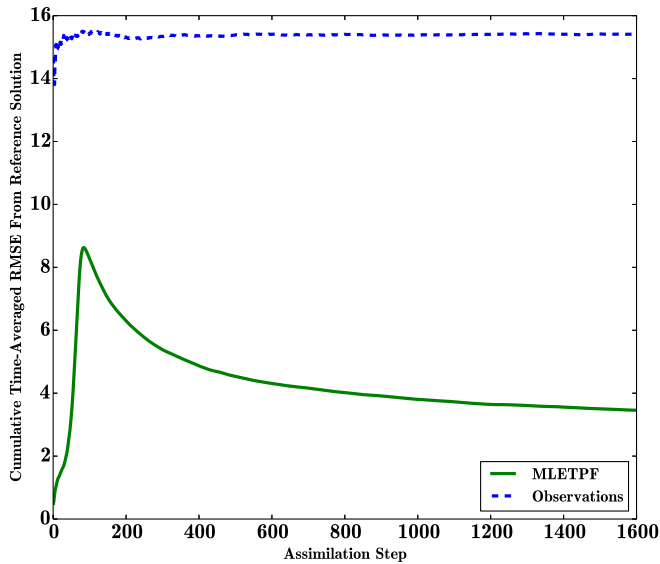}
  \caption{\textit{The cumulative time-averaged RMSE of the observations and MLETPF estimator away from the reference solution, $X_{t_{k}}'$, for the Stochastic Lorenz-96 Equations.}}
  \label{figure:Experiment2_StabilityofError}
\end{figure}

\begin{figure}[ht]
  \centering
  \includegraphics[width=100mm]{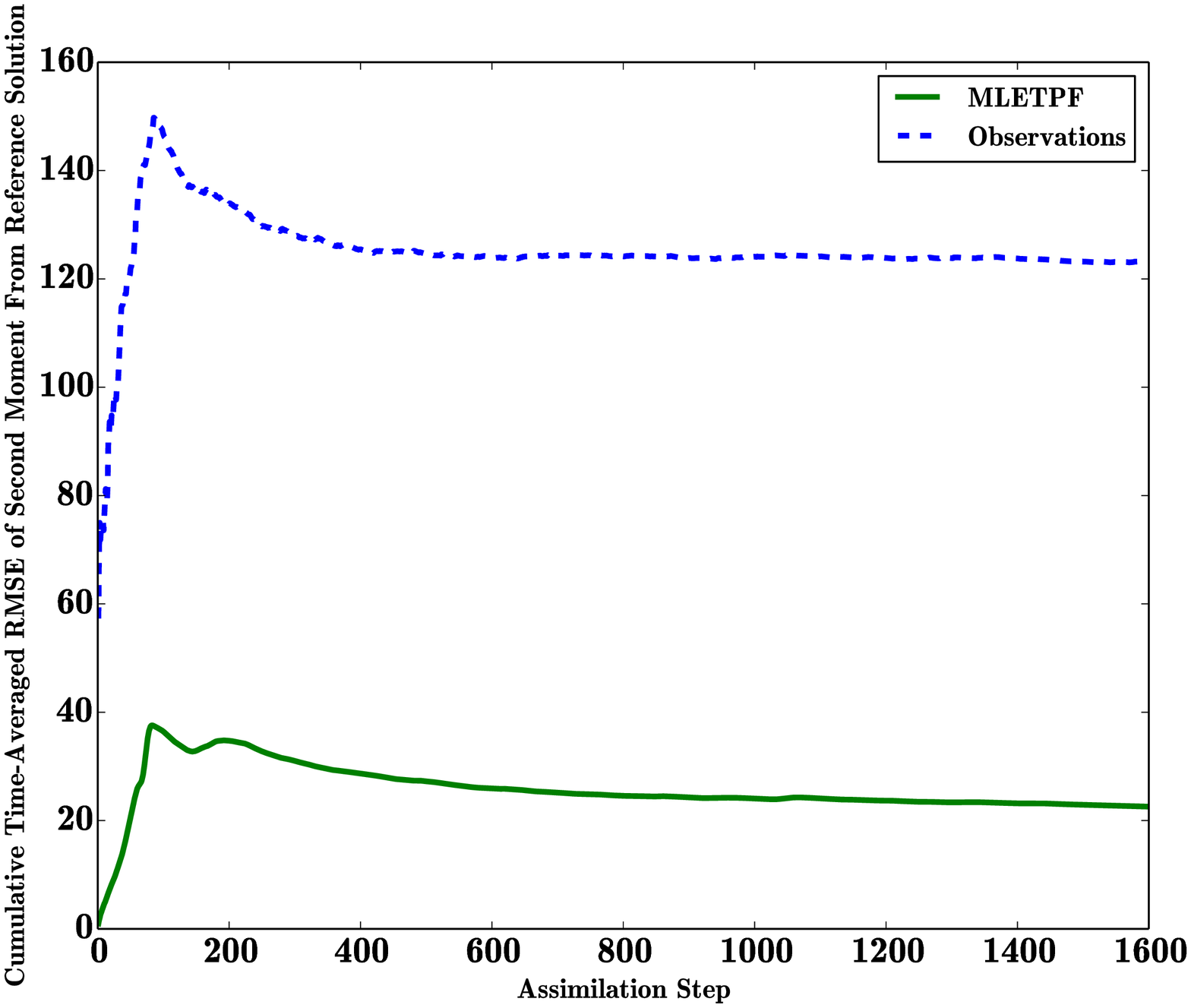}
  \caption{\textit{The cumulative time-averaged RMSE of the second moment of the observations and MLETPF estimator away from the second moment of the reference solution, $(X_{t_{k}}')^{2}$, for the Stochastic Lorenz-96 Equations.}}
  \label{figure:Experiment2_StabilityofErrorSecondMoment}
\end{figure}

\begin{figure}[ht]
  \centering
  \includegraphics[width=100mm]{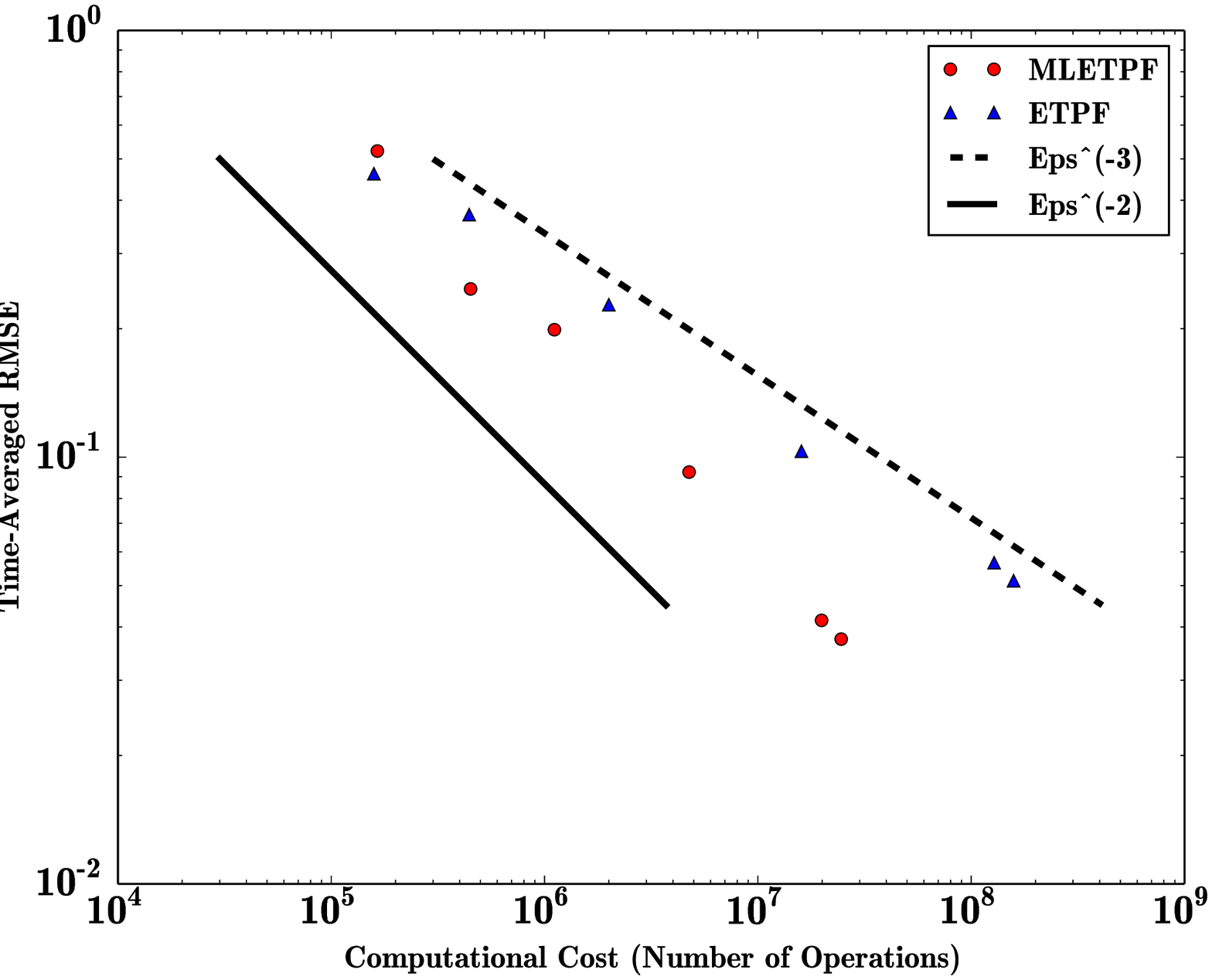}
  \caption{\textit{Computational Cost (number of operations) against the time-averaged RMSE of the ETPF and MLETPF estimators for the Stochastic Lorenz-96 Equations. Reference lines show the orders of decay of $RMSE^{-2}$ and $RMSE^{-3}$.}}
  \label{figure:Experiment2_CostVsAccuracy}
\end{figure}

Next, the stability of the MLETPF is considered. Since $N_{l}$ is fixed and does not change to bound error over time, we can look at the errors from the reference solution, $X_{t_{k}}'$, compared to the observational errors, to study the stability of the MLETPF estimator over time and check that the errors do not increase. One expects that, for a successful particle filter, the errors from the estimator should be less than the observational errors and remain stable. Figure \ref{figure:Experiment2_StabilityofError} shows this expected behaviour, where the cumulative time-averaged RMSE values from $X_{t_{k}}'$, of both the observations and the MLETPF estimator (using arbitrary values of $N_{0}=1000$ and $L=10$) are shown. For $k \in [1,N_{y}]$, these are defined to be,
\begin{equation}
\sqrt{\frac{1}{k}\sum^{k}_{i=1}\norm*{Y_{t_{i}}-X_{t_{i}}'}^{2}}
\end{equation}
for the observations and,
\begin{equation}
\sqrt{\frac{1}{k}\sum^{k}_{i=1}\norm*{\bar{X}_{L,t_{i}}-X_{t_{i}}'}^{2}}
\end{equation}
for the MLETPF estimator. To demonstrate the stability of the variance of the particle filter, cumulative time-averaged RMSE values for the second moments of $\bar{X}_{L,t_{k}}$ and $Y_{t_{k}}$ are shown in Figure \ref{figure:Experiment2_StabilityofErrorSecondMoment}. Finally, to compare the MLETPF with it's standard counterpart for this set of equations, Figure \ref{figure:Experiment2_CostVsAccuracy} shows the computational cost against the accuracy (RMSE) for the standard ETPF and the MLETPF estimators over varying values of $\epsilon$. Once again, one sets $N_{0}=\epsilon^{-2}$ for the MLETPF and $N=\epsilon^{-2}$ for the single level ETPF estimator. This follows the successful cost reductions achieved in the last example, defined in Theorem \ref{theorem:MLMC MSE}.

\section{Summary and outlook}
\label{sec:summary}

This paper has demonstrated a proof of concept for the application of MLMC to nonlinear filtering. The Ensemble Transform Particle Filter (ETPF), coupled with localisation, allows one to simply and cheaply carry out a multilevel coupling between each fine and coarse ensemble in each independent Monte Carlo estimator in the MLMC framework. A recent study has also proposed a framework to apply MLMC to nonlinear filtering with a modified random resampling step in the standard particle filtering methodology to couple particles from coarse and fine levels \cite{Jasra}. In contrast, the coupling in the present paper is designed to minimise the Wasserstein distance between the distributions of these transformed ensembles (in the standard ETPF methodology), originally suggested in \cite{GilesBook}. It has been shown through numerical experiments that one can restore positive correlation between fine and coarse ensembles which might have been lost if they had been transformed independently of one another. 
This in turn satisfies the neccessary constraints on the sample variance of each independent multilevel estimator, allowing the proposed MLETPF method to reduce the computational cost of the propagation of particles in the localised ETPF method. 

In general, localisation with $r_{loc,c}=0$ makes the computational cost of this coupling and the ensemble transform in the MLETPF cheap enough  that the multilevel framework can return overall computational cost reductions from the standard ETPF methods; the aim of the paper. It must be noted that although this paper has only touched on the case where the very crude $r_{loc,c}=0$ localisation is considered, due to small model cost test problems, this method could also be applied to high-dimensional systems where the model cost dominates that of the multivariate optimal transportation. One can do this without any crude constraints on $r_{loc,c}$ using the full multivariate coupling methodology presented in this paper and demonstrated numerically. However whether the variance decay of $\mathbb{V}_{l}$ from such a multivariate coupling would hold, producing as strong results, in the limit of $d>>N$ is unknown and thus the issue of how one would adjust this coupling to be used alongside other values of $r_{loc,c}$ to reduce the dimensionality of these multivariate couplings remains to be explored. \\
Iterative and approximate schemes for solving discrete optimal transportation problems have been an area of rapid research in the last few years \cite{Ravi} and this offers the chance to improve the multilevel coupling in the proposed method by reducing computational cost. This could be done by trading-off between the optimality and computational cost of the coupling for each $l$, e.g. more expensive / optimal couplings for greater $l$ with lower sample sizes $N_{l}$. \\
The form of the coupling used in this paper is simple to implement and has the potential to be used in plenty of applications, in and outside of data assimilation, whenever one wishes to establish consistent correlation between two distributions for variance reduction. Considering an extension for the multidimensional example presented in this paper, one could also apply a spatial multilevel framework, setting the spatial resolution ($\Delta X_{l}$) to be dependent on the level of discretization, as done in \cite{Cliffe,Barth} to gain even more significant cost reductions.

\bibliographystyle{siam}
\bibliography{refs}

\begin{thebibliography}{10}

\bibitem{Anderson}
{\sc J.~L. Anderson}, {\em Localization and sampling error correction in
  ensemble kalman filter data assimilation}, Monthly Weather Review, 140
  (2012), pp.~2359--2371.

\bibitem{Barth}
{\sc A.~Barth and A.~Lang}, {\em Multilevel {M}onte {C}arlo method with
  applications to stochastic partial differential equations}, International
  Journal of Computer Mathematics, 89 (2012), pp.~2479--2498.

\bibitem{Beskos}
{\sc A.~Beskos, A.~Jasra, K.~Law, R.~Tempone, and Y.~Zhou}, {\em Multilevel
  {S}equential {M}onte {C}arlo {S}amplers}, arXiv preprint arXiv:1503.07259,
  (2015).

\bibitem{Cappe}
{\sc O.~Capp{\'e}, S.~J. Godsill, and E.~Moulines}, {\em An overview of
  existing methods and recent advances in {s}equential {M}onte {C}arlo},
  Proceedings of the IEEE, 95 (2007), pp.~899--924.

\bibitem{Cheng}
{\sc Y.~Cheng and S.~Reich}, {\em A {M}c{K}ean optimal transportation
  perspective on {F}eynman-{K}ac formulae with application to data
  assimilation}, arXiv preprint arXiv:1311.6300,  (2013).

\bibitem{Cliffe}
{\sc K.~A. Cliffe, M.~B. Giles, R.~Scheichl, and A.~L. Teckentrup}, {\em
  Multilevel {M}onte {C}arlo methods and applications to elliptic {PDE}s with
  random coefficients}, Computing and Visualization in Science, 14 (2011),
  pp.~3--15.

\bibitem{Doucet}
{\sc A.~Doucet and A.~M. Johansen}, {\em A tutorial on particle filtering and
  smoothing: fifteen years later}, in Oxford Handbook of Nonlinear Filtering,
  2011.

\bibitem{GilesBook}
{\sc M.~Giles}, {\em Multilevel {M}onte {C}arlo {m}ethods}, in {M}onte {C}arlo
  and {Q}uasi-{M}onte {C}arlo {M}ethods, vol.~65, Springer, 2013, pp.~83--103.

\bibitem{Giles}
{\sc M.~B. Giles}, {\em Multilevel {M}onte {C}arlo {P}ath {S}imulation}, Oper.
  Res., 56 (2008), pp.~607--617.

\bibitem{GilesQuasi}
{\sc M.~B. Giles and B.~J. Waterhouse}, {\em Multilevel quasi-{M}onte {C}arlo
  path simulation}, Advanced Financial Modelling, Radon Ser. Computat. and
  Appl. Math., 8 (2009), pp.~165--181.

\bibitem{Hoel}
{\sc H.~Hoel, K.~J.~H. Law, and R.~Tempone}, {\em Multilevel ensemble {K}alman
  {F}iltering}, arXiv preprint arXiv:1502.06069,  (2015).

\bibitem{Jasra}
{\sc A.~Jasra, K.~Kamatani, K.~J.~H. Law, and Y.~Zhou}, {\em Multilevel
  particle filter}, arXiv preprint arXiv:1510.04977,  (2015).

\bibitem{Teckentrup}
{\sc C.~Ketelsen, R.~Scheichl, and A.~L. Teckentrup}, {\em A {H}ierarchical
  {M}ultilevel {M}arkov {C}hain {M}onte {C}arlo {A}lgorithm with {A}pplications
  to {U}ncertainty {Q}uantification in {S}ubsurface {F}low}, Center for Applied
  Scientific Computing,  (2013).

\bibitem{Legland}
{\sc F.~Le~Gland, V.~Monbet, and V.~Tran}, {\em Large sample asymptotics for
  the ensemble kalman filter}, The Oxford Handbook of Nonlinear Filtering,
  (2011), pp.~598--631.

\bibitem{Hungarian}
{\sc J.~Munkres}, {\em Algorithms for the assignment and transportation
  problems}, Journal of the Society for Industrial and Applied Mathematics, 5
  (1957), pp.~32--38.

\bibitem{FastEMD}
{\sc O.~Pele and M.~Werman}, {\em Fast and robust earth mover's distances}, in
  Computer Vision, IEEE 12th International Conference, 2009, pp.~460--467.

\bibitem{Ravi}
{\sc R.~Ravi}, {\em Iterative {M}ethods in {C}ombinatorial {O}ptimization}, in
  STACS'12 (29th Symposium on Theoretical Aspects of Computer Science),
  vol.~14, 2012, pp.~24--24.

\bibitem{Rebeschini}
{\sc P.~Rebeschini and R.~Van~Handel}, {\em Can local particle filters beat the
  curse of dimensionality?}, arXiv preprint arXiv:1301.6585,  (2013).

\bibitem{Reich}
{\sc S.~Reich}, {\em A nonparametric ensemble transform method for {B}ayesian
  inference}, SIAM J. Sci. Comput., 35 (2013), pp.~A2013--A2024.

\bibitem{EnsembleForecasting}
{\sc S.~Reich and C.~J. Cotter}, {\em Probabilistic {F}orecasting and
  {B}ayesian {D}ata {A}ssimilation}, Cambridge Univ. Press, 2015.

\bibitem{Schefzik}
{\sc R.~Schefzik, T.~L. Thorarinsdottir, and T.~Gneiting}, {\em Uncertainty
  quantification in complex simulation models using ensemble copula coupling},
  Statist. Sci., 28 (2013), pp.~616--640.

\bibitem{Villani}
{\sc C.~Villani}, {\em Optimal transport: {O}ld and {N}ew}, vol.~338, Springer
  Science \& Business Media, 2008.

\end{thebibliography}

\end{document}